\documentclass[a4paper,12pt]{amsart}
\usepackage{mathtools}
\usepackage[utf8]{inputenc} 
\usepackage[T1]{fontenc}
\usepackage[english]{babel} 
\usepackage{comment}
\usepackage{varioref}
\usepackage{amsmath}
\usepackage{amsthm}
\usepackage{amssymb}
\usepackage{stmaryrd}
\usepackage{esint}
\PassOptionsToPackage{normalem}{ulem}
\usepackage{ulem}
\usepackage{color}
\usepackage[all]{xy}
\usepackage{lineno,hyperref} 
\usepackage{amscd}
\usepackage{amsfonts}
\usepackage{enumerate}
\usepackage{graphicx}
\usepackage{pstricks}
\makeatletter

\newcommand\specialsectioning{\setcounter{secnumdepth}{-2}}

\DeclareFontEncoding{LGR}{}{}

\ProvideTextCommand{\~}{LGR}[1]{\char126#1}

\DeclareTextSymbolDefault{\textquotedbl}{T1}

\theoremstyle{plain}
\newtheorem{them}{Theorem}
\theoremstyle{plain}
\newtheorem{thm}{\protect\theoremname}[section]
\theoremstyle{plain}
\newtheorem*{thm*}{\protect\theoremname}
\theoremstyle{plain}

\theoremstyle{plain}
\newtheorem{lem}[thm]{\protect\lemmaname}
\theoremstyle{definition}
\newtheorem{defn}[thm]{\protect\definitionname}
\theoremstyle{definition}
\newtheorem{rem}[thm]{\protect\remarkname}
\theoremstyle{plain}
\newtheorem{prop}[thm]{\protect\propositionname}
\theoremstyle{definition}
\newtheorem*{rem*}{\protect\remarkname}
\theoremstyle{plain}
\newtheorem{cor}[thm]{\protect\corollaryname}
\theoremstyle{plain}
\newtheorem*{fact*}{\protect\factname}
\theoremstyle{definition}

\theoremstyle{definition}

\newcommand{\Z}{\mathbb{Z}}

\newcommand{\X}{\mathcal{X}}
\newcommand{\Y}{\mathcal{Y}}
\newcommand{\R}{\mathbb{R}}

\DeclareMathOperator{\RP}{RP}
\DeclareMathOperator{\PM}{P_T}
\DeclareMathOperator{\PMM}{P_{T\times T}}
\DeclareMathOperator{\pitk}{\pi_k^{top}}
\DeclareMathOperator{\cktop}{C_k^{top}}
\DeclareMathOperator{\ckme}{C_k^{meas}}
\DeclareMathOperator{\OC}{\overline{\mathcal{O}}}
\DeclareMathOperator{\id}{Id}
\DeclareMathOperator{\C}{C}
\DeclareMathOperator{\T}{T}
\DeclareMathOperator{\Homeo}{Homeo}
\newcommand{\cfnk}{\text{CF-Nil}$(k)$ }
\newcommand{\cfnkp}{\text{CF-Nil}($k$).}
\newcommand{\cfnkm}{\text{CF-Nil}($k-1$) }
\newcommand{\cfnkmp}{\text{CF-Nil}($k-1$).}
\newcommand{\cfnone}{\text{CF-Nil}($1$)}
\newcommand{\cfni}{\text{CF-Nil}($i$) }

\newenvironment{lyxlist}[1]
	{\begin{list}{}
		{\settowidth{\labelwidth}{#1}
		 \setlength{\leftmargin}{\labelwidth}
		 \addtolength{\leftmargin}{\labelsep}
		 }}
	{\end{list}}
\providecommand*{\strong}[1]{\textbf{#1}}

\newcommand\blfootnote[1]{ 
\begingroup 
\renewcommand\thefootnote{}\footnote{#1}
\addtocounter{footnote}{-1}
\endgroup 
}

\usepackage{babel}
\providecommand{\corollaryname}{Corollary}
  \providecommand{\definitionname}{Definition}
  \providecommand{\factname}{Fact}
  \providecommand{\lemmaname}{Lemma}
  \providecommand{\propositionname}{Proposition}
  \providecommand{\remarkname}{Remark}
\providecommand{\theoremname}{Theorem}

\newcommand{\h}{\hat}

\makeatother

\providecommand{\corollaryname}{Corollary}
\providecommand{\definitionname}{Definition}
\providecommand{\factname}{Fact}
\providecommand{\lemmaname}{Lemma}
\providecommand{\propositionname}{Proposition}
\providecommand{\remarkname}{Remark}
\providecommand{\theoremname}{Theorem}

\newcommand{\Addresses}{{
  \bigskip
  \footnotesize

   \textsc{Institute of Mathematics, Polish Academy of Sciences, ul. \'{S}niadeckich 8, 00-656 Warszawa, Poland.}\par\nopagebreak
  
  Yonatan Gutman: \texttt{y.gutman@impan.pl}
   
  \medskip

  \textsc{ School of Mathematical Sciences, Xiamen University, Xiamen, Fujian 361005, P.R. China;}
  
 Zhengxing Lian: \texttt{lianzx@mail.ustc.edu.cn; lianzx@xmu.edu.cn}

}}
\title[Pronilfactors and a topological Wiener-Wintner theorem]{Maximal pronilfactors and a topological Wiener-Wintner theorem}
\author{
 Yonatan Gutman \& Zhengxing Lian
}
\dedicatory{To Benjamin Weiss with great respect.}
\begin{document}

\maketitle

\begin{abstract}
  For strictly ergodic systems, we introduce the class of CF-Nil($k$) systems: systems
for which the maximal measurable and maximal topological $k$-step pronilfactors coincide as
measure-preserving systems. Weiss' theorem implies that such systems are abundant in a precise sense. We show that the \cfnk systems are precisely the class of minimal systems for which the $k$-step nilsequence version of the Wiener-Wintner average converges everywhere. As part of the proof we establish that pronilsystems are \textit{coalescent} both in the measurable and topological categories. In addition, we characterize a \cfnk system in terms of its \textit{$(k+1)$-th dynamical cubespace}. In particular, for $k=1$, this provides for strictly ergodic systems a new condition equivalent to the property that every measurable eigenfunction has a continuous version. 

\end{abstract}

\blfootnote{The authors were partially supported by the National Science Centre (Poland) grant 2016/22/E/ST1/00448. Y.G. was partially supported by the National Science Centre (Poland) grant 2020/39/B/ST1/02329. Z.L. was partially supported by the Xiamen Youth Innovation Foundation No. 3502Z20206037; the presidential research fund of Xiamen University No. 20720210034 and NNSF of China No. 1210010472.}



\tableofcontents{}\label{content}
\section{Introduction.}\label{sec:Introduction}
In recent years there has been an increase in interest in pronilfactors both for measure-preserving systems (m.p.s.) and topological dynamical systems (t.d.s.).  Pronilfactors of a given system are either measurable or topological (depending on the category) factors given by an inverse limit of nilsystems. A t.d.s. (m.p.s.) is called a topological (measurable) \textit{$d$-step pronilsystem}
if it is a topological (measurable) inverse limit of nilsystems of degree at most $d$.\footnote{
It is a classical fact that every (measurable) ergodic $d$-step pronilsystem is isomorphic as m.p.s. to a (topological) minimal $d$-step pronilsystem.} In the theory of measure preserving systems $(X,\mathcal{X},\mu,T)$ maximal measurable pronilfactors appear in connection with the $L^2$-convergence of the nonconventional ergodic averages
\begin{equation}\label{eq: convergence of multiple ergodic averages}
\frac{1}{N}\sum f_{1}(T^{n}x)\ldots f_{k}(T^{kn}x)
\end{equation}
for $f_1,\ldots,f_k\in L^\infty(X,\mu)$ (\cite{HK05,Z07}). 
In the theory of topological dynamical systems maximal topological pronilfactors appear in connection with the higher order regionally proximal relations (\cite{HKM10,SY12,GGY2018}).

When a system possesses both measurable and topological structure, it seems worthwhile  to investigate pronilfactors both from a measurable and topological point of view. A natural meeting ground are strictly ergodic systems - minimal topological dynamical systems $(X,T)$ possessing a unique invariant measure $\mu$. For $k\in \mathbb{Z}$ let us denote by $(Z_k(X),\mathcal{Z}_k(X),\mu_k,T)$ respectively $(W_k(X),T)$ the  maximal $k$-step measurable respectively topological pronilfactor\footnote{Both these objects exist and are unique in a precise sense. See Subsection \ref{subsec:Universality}.} of $(X,T)$. Clearly $(W_k(X),T)$ has a unique invariant measure $\omega_k$. We thus pose the question when is $(W_k(X), \mathcal{W}_k(X), \omega_k,T)$ isomorphic to $(Z_k(X),\mathcal{Z}_k(X),\mu_k,T)$ as m.p.s.? We call a t.d.s. which is strictly ergodic and for which $(W_k(X), \mathcal{W}_k(X), \omega_k,T)$ is isomorphic to $(Z_k(X),\mathcal{Z}_k(X),\mu_k,T)$ as m.p.s., a \textit{\cfnk} system\footnote{This terminology is explained in Subsection \ref{subsec: def:continuously to k-nil}.}.
Note that $(W_k(X), \mathcal{W}_k(X), \omega_k,T)$ is always a measurable factor of $(Z_k(X),\mathcal{Z}_k(X),\mu_k,T)$. At first glance it may seem that \cfnk systems are rare however a theorem by Benjamin Weiss regarding topological models for  measurable extensions implies that every ergodic m.p.s. is measurably isomorphic to a \cfnk system\footnote{See Subsection \ref{subsec:Weiss thm}.}.

We give two characterizations of \cfnk systems. The first characterization is related to the Wiener-Wintner theorem while the second characterization is related to \textit{$k$-cube uniquely ergodic} systems - a class of topological dynamical systems introduced in \cite{GL2019}.

The Wiener-Wintner theorem (\cite{WW1941}) states that for an ergodic system $(X,\mathcal{X},\mu,T)$, for $\mu$-a.e. $x\in X$, \textit{any} $\lambda\in \mathbb{S}^1$ and any $f\in L^\infty(\mu)$, the following limit exists:
\begin{equation}\label{WW theorem converge}\lim_{N\rightarrow \infty}\frac{1}{N}\sum_{n=1}^N \lambda^nf(T^nx)
\end{equation}
Denote by $M_T\subset \mathbb{S}^1$ the set of measurable eigenvalues\footnote{Measurable and topological eigenvalues are defined in Subsection \ref{subsec:Dynamical-background.}.} of $(X,\mathcal{X},\mu,T)$. Let $P_{\lambda}f$ be the projection of $f$ to the eigenspace corresponding to $\lambda$ (in particular for $\lambda\notin M_T$, $P_{\lambda}f\equiv 0$). For fixed $\lambda\in \mathbb{S}^1$, one can show  \eqref{WW theorem converge} converges a.s. to $P_{\lambda}f$.

In  \cite{Lesigne1990} Lesigne proved that a.s. convergence in \eqref{WW theorem converge} still holds when the term $\lambda^n$ is replaced by a (continuous function) of a real-valued polynomial $P(n)$, $P\in \mathbb{R}[t]$. In \cite{frantzikinakis2006uniformity} Frantzikinakis established a \textit{uniform version}\footnote{In the context of the Wiener-Wintner theorem, \textit{uniform versions} are a.s. convergence results involving a supremum over weights belonging to a given class. The first result of this type was obtained by Bourgain in \cite{B90}.} of this theorem. In \cite{HK09}, Host and Kra showed that a.s. convergence in \eqref{WW theorem converge} still holds when the term $\lambda^n$ is replaced by a \textit{nilsequence}. In \cite{EZK13} Eisner and Zorin-Kranich established a uniform version of this theorem.

For topological dynamical systems one may investigate the question of \textit{everywhere} convergence in the Wiener-Wintner theorem.
In \cite{Robinson1994}, Robinson proved that
 for an uniquely ergodic system $(X,\mu,T)$, for any $f\in C(X)$, if every measurable eigenfunction of $(X,\mathcal{X},\mu,T)$ has a continuous version then the limit \eqref{WW theorem converge} converges everywhere. He noted however that if $P_\lambda f\neq 0$ for some $\lambda \in M_T$, then the convergence of \eqref{WW theorem converge} is not uniform in $(x,\lambda)$, since the limit function $P_\lambda f(x)$ is not continuous on $X\times \mathbb{S}^1$.\footnote{Note $M_T$ is countable.} Moreover Robinson constructed a strictly ergodic system $(X,T)$ such that \eqref{WW theorem converge} does not converge for some continuous function $f\in C(X)$, some $\lambda\in \mathbb{C}$ and some $x\in X$. Other topological versions of the Wiener-Wintner theorem may be found in \cite{Assani1992,Fan2018}\footnote{One should also note that topological Wiener-Winter theorems have been investigated in the generality of operator semigroups by Schreiber and Bartoszek and Śpiewak (\cite{schreiber2014topological,bartoszek2017note}).}. 
 
 The first main result of this article is the following theorem: 
\begin{them}\label{TWWT main}
Let $(X,T)$ be a minimal system. Then for $k\geq 0$ the following are equivalent:
\begin{lyxlist}{0.0000}
     \item [(I).] $(X,T)$ is a \cfnk system.
     \item [(II).] For any $k$-step nilsequence $\{a(n)\}_{n\in \mathbb{Z}}$, any continuous function $f\in C(X)$ and any $x\in X$,
    \begin{equation}\label{TWWT limit}\lim_{N\rightarrow \infty}\frac{1}{N}\sum_{n=1}^Na(n)f(T^nx)
    \end{equation}
    exists.
\end{lyxlist}   
\end{them}

We remark that the direction (I)$\Rightarrow $(II) of Theorem \ref{TWWT main} follows from \cite{HK09} whereas the case $k=1$ of Theorem \ref{TWWT main} follows from \cite[Theorem 1.1]{Robinson1994}.

As part of the proof of Theorem \ref{TWWT main} we established a fundamental property for pronilsystems:

\begin{them}\label{thm:coalescense}
Let $(Y,\nu,T)$ be a minimal (uniquely ergodic) $k$-step pronilsystem. Then
\begin{lyxlist}{0.0000}
    \item [(I).]
$(Y,\nu,T)$ is measurably coalescent, i.e. if $\pi:(Y,\nu,T)\rightarrow (Y,\nu,T)$ is a measurable factor map, then $\pi$ is a measurable isomorphism.
\end{lyxlist}   
and 
\begin{lyxlist}{0.0000}
    \item [(II).]
 $(Y,T)$ is topologically coalescent, i.e. if $\Phi:(Y,T)\rightarrow (Y,T)$ is a topological factor map, then $\Phi$ is a topological isomorphism.
 \end{lyxlist}   
\end{them}

As part of the the theory of higher order regionally proximal relations, Host, Kra and Maass introduced in \cite{HKM10} the \textit{dynamical cubespaces}
$\C_{\T}^{n}(X)\subset X^{2^n}$, $n\in \mathbb{N}:=\{1,2,\ldots\}$. These compact sets enjoy a natural action by the \textit{Host-Kra cube groups} $\mathcal{HK}^{n}(T)$. According to the terminology introduced in \cite{GL2019}, a t.d.s. $(X,T)$ is called \textit{$k$-cube uniquely ergodic} if 
$(\C_{\T}^{k}(X),\mathcal{HK}^{k}(T))$
is uniquely ergodic. The third main result of this article is the following theorem:
\begin{them}
\label{thm:equiv cube u.e}
Let $(X,T)$ be a minimal t.d.s.  Then the following are equivalent for any $k\geq 0$:
\begin{lyxlist}{0.0000}
     \item [(I).] $(X,T)$ is a \cfnk system.
     \item [(II).] $(X,T)$ is
$(k+1)$-cube uniquely ergodic.
\end{lyxlist}  
\end{them}
We remark that the direction (I) $\Rightarrow $ (II) follows from \cite{HSY2017}. \\

In the context of  various classes of  strictly ergodic systems, several authors have investigated the question of whether every measurable eigenfunction has a continuous version. Famously in \cite{Host1986} (see also \cite[Page 170]{M2010}), Host established this is the case for
 \textit{admissible substitution dynamical systems}. In \cite[Theorem 27]{BDM10} an affirmative answer was given for  strictly ergodic \textit{Toeplitz type systems of finite rank}. In \cite{DFM2019}, the continuous and measurable eigenvalues of minimal Cantor systems were studied.

It is easy to see that for strictly ergodic systems $(X,T)$ the condition that every measurable eigenfunction has a continuous version is equivalent to the fact that $(X,T)$ is \cfnone.
Thus Theorem \ref{thm:equiv cube u.e} provides  for strictly ergodic systems a new condition equivalent to the property that every measurable eigenfunction has a continuous version. Namely this holds iff $(\C_{\T}^{2}(X),\mathcal{HK}^{2}(T))$ is uniquely ergodic. As the last condition seems quite manageable one wonders if this new equivalence may turn out to be useful in future applications.\\

\noindent
\textbf{Structure of the paper.} In Subsections \ref{subsec:Dynamical-background.}--\ref{Subsec: Con exp} we review some definitions and classical facts; In Subsections \ref{subsec:Nilsystem-and-a system at most d}--\ref{subsec: def:continuously to k-nil}, we introduce the topological and measurable maximal pronilfactors and define the \cfnk systems; In Subsection \ref{subsec:Weiss thm}, we use Weiss's Theorem to show that the \cfnk systems are abundant; In Section \ref{uniqueness for Z and W}, we prove Theorem \ref{thm:coalescense} and then establish \textit{universality} for maximal pronilfactors; In Section \ref{sec:cubespace}, we prove Theorem \ref{thm:equiv cube u.e}; In Section \ref{sec TWWT}, we prove Theorem \ref{TWWT main}.\\

\noindent
\textbf{Acknowledgements.} We are grateful to Bernard Host, Mariusz Lemańczyk and and anonymous referee for helpful comments.  

\section{Preliminaries.}\label{sec:Preliminaries}

\subsection{\label{subsec:Dynamical-background.}Dynamical systems.} 
Throughout this article we assume every topological space to be metrizable. A $\Z$-\strong{topological dynamical system (t.d.s.)} is a pair $(X,T)$,
where $X$ is a compact space and $T$ is a homeomorphism on $X$. Denote by $C(X)$ the set of real-valued continuous functions on $X$. The \strong{orbit} $\mathcal{O}(x)$ of $x\in X$ is the set $\mathcal{O}(x)=\{T^n x:n\in \Z\}$. Its closure is denoted by $\OC(x)$ A t.d.s.  is \strong{minimal} if $\OC(x)=X$ for all $x\in X$. A t.d.s.  $(X,T)$ is \strong{distal} if for a compatible metric $d_{X}$
of $X$, for any $x\neq y\in X$, $\inf_{n\in \Z}d_{X}(T^n x,T^n y)>0$. We say $\pi:(Y,S)\rightarrow(X,T)$ is a \strong{topological factor map} if $\pi$ is a continuous and
surjective map such that for any $x\in X$, $\pi(S x)=T\pi(x)$. Given such a map, $(X,T)$ is called a \strong{topological factor} of $(Y,S)$ and $(X,T)$ is said to \strong{factor continuously} on $(Y,S)$. If in addition $\pi$ is injective then it is called a \strong{topological isomorphism} and $(Y,S)$ and $(X,T)$ are said to be \textbf{isomorphic as t.d.s.} 
A factor map $\pi:(Y,S)\rightarrow(X,T)$ is called a \strong{topological group extension} by a compact group $K$ if there exists a continuous action $\alpha:K\times Y\rightarrow Y$ such that the actions $S$ and $K$ commute and for all $x,y\in Y$, $\pi(x)=\pi(y)$ iff there
exists a unique $k\in K$ such that $kx = y$. 
A \textbf{(topological) eigenvalue} of a t.d.s. $(X,T)$ is a complex number $\lambda \in \mathbb{S}^1$ such that an equation of the form $f(Tx)=\lambda f(x)$ holds for some $f\in C(X,\mathbb{C})$ and all $x\in X$. The function $f$ is referred to as a \textbf{continuous} or \textbf{topological eigenfunction}.  

Let $\{(X_{m},T_m)\}_{m\in\mathbb{N}}$ be a sequence of t.d.s. and for any $m\geq n$,  $\pi_{m,n}:(X_{n},T_n)\rightarrow(X_{m},T_m)$ factor maps such
that $\pi_{i,l}=\pi_{i,j}\circ\pi_{j,l}\text{ for all }1\leq i\leq j\leq l.$ The \strong{inverse limit} of $\{(X_{m},T_m)\}_{m\in\mathbb{N}}$ is defined to be the system $(X,T)$, where

\[
X=\{(x_{m})_{m\in \mathbb{N}}\in\prod_{m\in\mathbb{N}}X_{m}:\ \pi_{m+1}(x_{m+1})=x_{m}\text{ for }m\geq1\}
\]
equipped with the product topology and $T(x_{m})_{m\in\mathbb{N}}\triangleq(T_mx_{m})_{m\in\mathbb{N}}$. We write
$(X,T)=\underleftarrow{\lim}(X_{m},T_m)$.

 A \strong{measure preserving probability system (m.p.s.)} is a quadruple $(X,\mathcal{X},\mu,T)$, where $(X,\mathcal{X},\mu)$
is a standard Borel probability space (in particular $X$ is a Polish space and $\mathcal{X}$ is its Borel $\sigma$-algebra) and $T$ is an invertible Borel measure-preserving map ($\mu(TA)=\mu(A)$ for all $A\in \mathcal{X}$). An m.p.s. $(X,\mathcal{X},\mu,T)$ is \strong{ergodic}
if for every set $A\in\mathcal{X}$ such that $T(A)=A$,
one has $\mu(A)=0$ or $1$. A \strong{measurable factor map} is a Borel map  $\pi:(X,\mathcal{X},\mu,T)\rightarrow (Y,\mathcal{Y},\nu,S)$ which is induced by a $G$-invariant sub-$\sigma$-algebra of $\mathcal{X}$ (\cite[Chapter 2.2]{G03}). Given such a map, $(Y,\mathcal{Y},\nu,S)$ is called a \strong{measurable factor} of $(X,\mathcal{X},\mu,T)$. If $\pi$ is in addition invertible on a set of full measure then $\pi$ is called a \strong{measurable isomorphism} and $(X,\mathcal{X},\mu,T)$ and $(Y,\mathcal{Y},\nu,S)$ are said to be \textbf{isomorphic as m.p.s.} 
Let $(Y,\mathcal{Y},\nu,S)$ be an m.p.s. and $A$ a compact group with Borel $\sigma$-algebra $\mathcal{A}$ and Haar measure $m$. A \textbf{skew-product} $(Y\times A,\mathcal{Y}\otimes \mathcal{A}, \nu\times m,T)$ is given by the action $T(y,u)=(Sy,\beta(y)u)$, where $\beta:Y\rightarrow A$ is a Borel map, the so-called \textit{cocycle} of the skew-product. The projection $(Y\times A,\mathcal{Y}\otimes \mathcal{A}, \nu\times m,T)\rightarrow (Y,\mathcal{Y},\nu,S)$ given by $(y,a)\mapsto y$ is called a \textbf{measurable group extension} (cf. \cite[Theorem 3.29]{G03}).

A \textbf{(measurable) eigenvalue} of a m.p.s. $(X,\mathcal{X},\mu,T)$ is a complex number $\lambda\in\mathbb{S}^1$ such that an equation of the form $f(Tx)=\lambda f(x)$ holds for $\mu$-a.e. $x\in X$ for some Borel function $f:X\rightarrow \mathbb{C}$. The function $f$ is referred to as a \textbf{measurable eigenfunction}.

Denote by $\PM(X)$ the set of $T$-invariant Borel probability measures of $X$. A t.d.s. $(X,T)$ is called \textbf{uniquely ergodic} if $|\PM(X)|=1$. If in addition it is minimal then it is called \textbf{strictly ergodic}. For a strictly ergodic system $(X,T)$ with a (unique) invariant measure $\mu$, we will use the notation $(X,\mu,T)$. When considered as a m.p.s. it is with respect to its Borel $\sigma$-algebra.

Occasionally in this article we will consider more general group actions than $\Z$-actions. Thus a $G$-\strong{topological dynamical system (t.d.s.)} is a pair $(G,X)$ consisting of a (metrizable) topological group $G$ acting on a (metrizable) compact space
$X$. For $g\in G$ and $x\in X$ we denote the action both by $gx$ and $g.x$. We will need the following proposition:
\begin{prop}\label{prop:factor of ue}
  Let $G$ be an amenable group.  Let $(G,X)$ be uniquely ergodic and let $(G,X)\rightarrow (G,Y)$ be a topological factor map. Then $(G,Y)$ is uniquely ergodic.
\end{prop}

\begin{proof}
See proof of Proposition 8.1 of \cite{angel2014random}.
\end{proof}

\subsection{Topological models.}
\begin{defn}\label{subsec:topological model}
Let $(X,\mathcal{X},\mu,T)$ be a m.p.s. We say that a t.d.s.  $(\hat{X},\hat{T})$
is a \textbf{topological model} for $(X,\mathcal{X},\mu,T)$ w.r.t. to a $\hat{T}-$invariant probability
measure $\hat{\mu}$ on $\hat{\mathcal{X}}$, the Borel $\sigma$-algebra of $X$,
if the system 
$(X,\mathcal{X},\mu,T)$ is isomorphic to $(\hat{X},\hat{\mathcal{X}},\hat{\mu},\hat{T})$ as m.p.s., that is, there exist a $T$-invariant Borel subset $C\subset X$ and  a $\hat{T}$-invariant Borel subset $\hat{C}\subset \hat{X}$ of full measure and a (bi)measurable and equivariant measure preserving bijective Borel map $p:C\rightarrow \hat{C}$. Notice that oftentimes in this article $(\hat{X},\hat{T})$ will be uniquely ergodic so that $\hat{\mu}$ will be the unique $\hat{T}-$invariant probability
measure of $X$.
\end{defn}

\begin{defn}\label{def:top model for map} \sloppy
Let $(X,\mathcal{X},\mu,T)$, $(Y,\Y, \nu, S)$ be m.p.s. Let  $(\hat{X},\hat{T})$, $(\hat{Y},\hat{S})$ be t.d.s. which are topological models of $(X,\mathcal{X},\mu,T)$ and $(Y,\Y, \nu, S)$ w.r.t. measures $\h{\mu}$ and $\h{\nu}$ as witnessed by maps $\phi$ and $\psi$ respectively. We say that $\h{\pi}: (\h{X},\h{T})\rightarrow (\h{Y},\h{S})$ is a \textbf{topological
model} for a factor map $\pi: (X,\X, \mu, T)\rightarrow (Y,\Y, \nu, S)$ if
$\h{\pi}$ is a topological factor and the following diagram
\[
\begin{CD}
X @>{\phi}>> \h{X}\\
@V{\pi}VV      @VV{\h{\pi}}V\\
Y @>{\psi }>> \h{Y}
\end{CD}
\]
is commutative, i.e. $\h{\pi}\phi=\psi\pi$
\end{defn}

\subsection{Conditional expectation.}\label{Subsec: Con exp}

Let  $(X,\mathcal{X},\mu)$ be a  probability space and let $\mathcal{B}$ be a sub-$\sigma$-algebra of $\mathcal{X}$. For $f\in L^1(\mu)$, the \strong{conditional expectation} of $f$ w.r.t. $\mathcal{B}$ is the unique function $\mathbb{E}(f|\mathcal{B})\in L^1(X,\mathcal{B},\mu)$ satisfying
\begin{equation}\label{eq:cond exp}
\int_Bfd\mu=\int_B\mathbb{E}(f|\mathcal{B})d\mu
\end{equation}
for every $B\in \mathcal{B} $. For $f\in L^1(\mu)$ and $g\in L^{\infty}(X,\mathcal{B},\mu)$, it holds (see \cite[Chapter 2, Section 2.4]{HK2018}):
\begin{equation}\label{conditional expection}
\int_Xfgd\mu=\int_X\mathbb{E}(f|\mathcal{B})gd\mu.
\end{equation}
Let $(X,\mathcal{X},\mu)$ and $(Y,\mathcal{Y},\nu)$ be probability spaces and let $\pi:X\rightarrow Y$ be a measurable map such that $\pi_{*}\mu=\nu$. Denote by $\mathbb{E}(f|Y)\in L^1(Y,\nu)$ the function such that $\mathbb{E}(f|Y)=\mathbb{E}(f|\pi^{-1}(\mathcal{Y}))\circ\pi^{-1}$. Note this is well-defined. Thus the difference between $\mathbb{E}(f|Y)$ and $\mathbb{E}(f|\pi^{-1}(\mathcal{Y}))$ is that the first function is considered as a function on $Y$ and the second as a function on $X$.

\subsection{Pronilsystems and nilsequences.}\label{subsec:Nilsystem-and-a system at most d}

\begin{defn}

A (real) \strong{Lie group} is a group that is also a finite dimensional
real smooth manifold such that the group operations of multiplication
and inversion are smooth. Let $G$ be a Lie group. Let $G_1=G$ and $G_{k}=[G_{k-1},G]$
for $k\geq 2$, where $[G,H]=\{[g,h]:g\in G,h\in H\}$ and $[g,h]=g^{-1}h^{-1}gh$.
If there exists some $d\geq1$ such that  $G_{d+1}=\{e\}$,
$G$ is called a \strong{$d$-step nilpotent} Lie group. We say that a discrete subgroup $\Gamma$ of a Lie group $G$ is 
\strong{cocompact} if $G/\Gamma$, endowed with the quotient topology,
is compact. We say that quotient $X=G/\Gamma$ is a \strong{$d$-step nilmanifold}
if $G$ is a $d$-step nilpotent Lie group and
$\Gamma$ is a discrete, cocompact subgroup. The nilmanifold $X$ admits a natural action by $G$ through \textbf{translations} $g.a\Gamma=ga\Gamma$, $g,a \in G$. The \textbf{Haar measure} of $X$ is the unique Borel probability measure on $X$ which is invariant under this action. A \strong{nilsystem of degree at most $d$}, $(X,T)$, is given by an $d$-step nilmanifold
$X=G/\Gamma$ and $T\in G$ with action $T.a\Gamma=Ta\Gamma$. When a nilsystem is considered as a m.p.s. it is always w.r.t. its Haar measure.
\end{defn}

\begin{defn}\label{def:pronilsystem}
A t.d.s.  (m.p.s) is called a topological (measurable) \strong{$d$-step pronilsystem} 
if it is a topological (measurable) inverse limit of nilsystems of degree at most $d$. By convention a $0$-step pronilsystem is the one-point trivial system.

\end{defn}

\begin{rem}\label{rem:meas and top pronil}
 \sloppy By \cite[p. 233]{HK2018} if an ergodic measurable $d$-step pronilsystem is presented as the inverse limit $(X,\X,\nu,T)=\underleftarrow{\lim}(X_{m},\X_m,\nu_m,T_m)$ given by the measurable factor maps  $\pi_m:(X_{m},\X_m,\nu_m,T_m)\rightarrow (X_{m-1},\X_{m-1},\nu_{m-1},T_{m-1})$ between nilsystems  of degree at most $d$ then there exist topological factor maps $\tilde{\pi}_m:(X_{m},T_m)\rightarrow (X_{m-1},T_{m-1})$ such that $\tilde{\pi}=\pi$ $\nu_m$-a.e. and so effectively one can consider  $(X,\X,\nu,T)$ as a (minimal) topological pronilsystem. Moreover any two $d$-step pronilsystem topological models of $(X,\X,\nu,T)$ are isomorphic as t.d.s. (Theorem \ref{thm: meas coalescence}).

\end{rem}

\begin{defn}(\cite[Definition 2.2]{HKM10})\label{def:nilsequence}
A bounded sequence $\{a(n)\}_{n\in \mathbb{Z}}$ is called a \strong{$d$-step nilsequence} if there exists a $d$-step pronilsystem $(X,T)$, $x_0\in X$ and a continuous function $f\in C(X)$ such that $a(n)=f(T^nx_0)$ for $n\in \mathbb{Z}$. 
\end{defn}

\begin{thm}(\cite[Theorem 3.1]{HK09})\label{thm:basic properties for nilsystems}
Let $(X,T)$  be a nilsystem. Then
$(X,T)$ is uniquely ergodic if and only if $(X,T)$ is ergodic w.r.t. the Haar measure if and only if $(X,T)$ is minimal.
\end{thm}

The following proposition is an immediate corollary of the previous theorem.

\begin{prop}\label{prop:basic properties for pronilsystems}
Let $(X,T)$ be a pronilsystem. Then
$(X,T)$ is uniquely ergodic if and only if $(X,T)$ is minimal.
\end{prop}

\begin{defn}
Let $(X,\mu,T)$ be a strictly ergodic t.d.s. We say that a t.d.s. $(Y,T)$ is a \textbf{topological $k$-step pronilfactor} of $(X,T)$ if it is a topological factor of $(X,T)$ and if it is isomorphic to a $k$-step pronilsystem as a t.d.s. We say that a m.p.s. $(Y,\mathcal{Y},\nu,T)$ is a \textbf{measurable $k$-step pronilfactor} of $(X,T)$ if it is a measurable factor of $(X,\mathcal{X},\mu,T)$ and if it is isomorphic to a $k$-step pronilsystem as a m.p.s. 
\end{defn}

\subsection{Host-Kra structure theory machinery.}\label{subsec:Definition-of-the mu k}

 By a \strong{face} of the discrete cube $\{0,1\}^{k}$ we mean a
subcube obtained by fixing some subset of the coordinates. For $k\in \mathbb{N}$, 
let $[k]=\{0,1\}^k$. Thus $X^{[k]}=X\times\cdots\times X$, $2^k$ times and similarly $T^{[k]}=T\times\cdots\times T$, $2^k$ times. For $x\in X$, $x^{[k]}=(x,\ldots, x)\in X^{[k]}$.
Let $[k]_*=\{0,1\}^k\setminus \{\vec{0}\}$ and define $X_{*}^{[k]}=X^{[k]_*}$.

 \begin{defn} (\cite{HK05}) \label{def:HK cube group and face group}
Let $(X, \mathcal{X},\mu,T)$ be an ergodic m.p.s. For $1\leq j\leq k$, let $\overline{\alpha}_{j}=\{v\in\{0,1\}^{k}:v_j=1\}$ be
the \textbf{$j$-th upper face} of $\{0,1\}^{k}$. For any face $F\subset \{0,1\}^k$, define $$(T^{F})_v=\begin{cases} T & v\in F\\
\id & v\notin F.
\end{cases}$$
Define the \strong{face group}   $\mathcal{F}^k(T)\subset \Homeo(X^{[k]})$ to be the group generated by the elements $\{T^{\overline{\alpha}_{j}}:1\leq j\leq k\}$. Define the the $k$-th \strong{Host-Kra cube group} $\mathcal{HK}^k(T)$ to be the subgroup of $\Homeo(X^{[k]})$ generated by $\mathcal{F}^k(T)$ and $T^{[k]}$.
\end{defn}
\begin{defn} (\cite{HK05}) \label{def: muk definition} Let $(X,\mathcal{B},\mu,T)$ be an ergodic m.p.s. Let $\mu^{[1]}=\mu\times \mu$. For $k\in \mathbb{N}$, let $\mathcal{I}_{T^{[k]}}$ be the $T^{[k]}$-invariant
$\sigma$-algebra of $(X^{[k]},\mathcal{X}^{[k]},\mu^{[k]})$.
Define $\mu^{[k+1]}$ to be the relative independent
joining of two copies of $\mu^{[k]}$ over $\mathcal{I}_{T^{[k]}}$.
That is, for $f_{v}\in L^{\infty}(\mu)$, $v\in\{0,1\}^{k+1}$:
\begin{multline*}
\int_{X^{[k+1]}}\prod_{v\in\{0,1\}^{k+1}}f_{v}(x_v)d\mu^{[k+1]}(x)=\\\int_{X^{[k]}}\mathbb{{E}}(\prod_{v\in\{0,1\}^{k}}f_{v0}|\mathcal{I}_{T^{[k]}})(x)\mathbb{{E}}(\prod_{v\in\{0,1\}^{k}}f_{v1}|\mathcal{I}_{T^{[k]}})(x)d\mu^{[k]}(x)
.
\end{multline*}
In particular, from Equation \eqref{conditional expection}, it follows that for all measurable functions $H_1,H_2\in L^\infty(X^{[k]},\mu^{[k]})$,
\begin{multline}\label{def muk 2}\int_{X^{[k]}}\mathbb{{E}}(H_1|\mathcal{I}_{T^{[k]}})(c)\mathbb{{E}}(H_2|\mathcal{I}_{T^{[k]}})(c)d\mu^{[k]}(c)=\\ \int_{X^{[k]}}\mathbb{{E}}(H_1|\mathcal{I}_{T^{[k]}})(c)H_2(c)d\mu^{[k]}(c).\end{multline}
Note $\mu^{[k]}$ is  $\mathcal{HK}^{k}(T)$-invariant (\cite[Chapter 9, Proposition 2]{HK2018}).
\end{defn}

\begin{defn}
\cite[Chapter 9, Section 1]{HK2018}\label{def:J} For $k\in \mathbb{N}$, let $\mathcal{J}_{*}^{k}$ be the $\sigma$-algebras of sets invariant
under $\mathcal{F}^{k}(T)$ on $X_{*}^{[k]}$.
\end{defn}

\begin{defn}
\label{def:Z_k} 
\cite[Subsection 9.1]{HK2018} Let $(X,\mathcal{X},\mu,T)$ be an ergodic m.p.s. For $k\in \mathbb{N}$, define $\mathcal{Z}_{k}(X)$ to be the $\sigma$-algebra consisting of measurable sets $B$ such that there exists a $\mathcal{J}_{*}^{k+1}$-measurable
set $A\subset X_*^{[k+1]}$ so that
up to $\mu^{[k+1]}$- measure zero it holds:
\begin{equation*}
X\times A=B\times X_{*}^{[k+1]}\label{eq:measurable unique completeness}
\end{equation*}
Define the \textbf{$k$-th Host-Kra factor} $Z_k(X)$  as the measurable factor of $X$ induced by $\mathcal{Z}_{k}(X)$ and denote by $\pi_k:X\rightarrow Z_{k}(X)$ the \textbf{(measurable) canonical $k$-th projection}. Let $\mu_k$ be the projection of $\mu$ w.r.t. $\pi_k$.
\end{defn}

\begin{defn}
\label{def:The-seminorms}  Let $(X,\mathcal{X},\mu,T)$ be an m.p.s. and $k\in \mathbb{N}$. The \textbf{Host-Kra-Gowers seminorms} on $L^{\infty}(\mu)$ are defined as follows:

\[
|||f|||_{k}=(\int\prod_{v\in\{0,1\}^{k}}\mathcal{C}^{|v|}fd\mu^{[k]})^{1/2^{k}},
\]
where $|(v_1,\ldots,v_{k+1})|=\Sigma_{i=1}^{k+1} v_i$ and $\mathcal{C}^{n}z=z$ if
$n$ is even and $\mathcal{C}^{n}z=\overline{z}$ if
$n$ is odd. By \cite[Subsection 8.3]{HK2018},
$|||\cdot|||_{k}$ is a seminorm.
\end{defn}

\begin{lem}\cite[Chapter 9, Theorem 7]{HK2018}\label{lem: seminorm and k factor} Let $(X,\mathcal{X},\mu,T)$
be an ergodic m.p.s. and $k\in\mathbb{N}$. Then for $f\in L^{\infty}(\mu)$,
$|||f|||_{k+1}=0$ if and only if $\mathbb{E}(f|\mathcal{Z}_{k}(X))=0$.
\end{lem}

\subsection{Maximal measurable  pronilfactors.}\label{subsec:classical meas max}
\begin{defn}
Let $k\in \mathbb{N}$. A m.p.s. 
$(X,\mathcal{X},\mu,T)$ is called a \strong{(measurable) system of order $k$} if it is isomorphic to $(Z_k(X),\mathcal{Z}_k(X),\mu_k,T)$.
\end{defn}

\begin{thm}\label{thm:meas order k}(\cite[Theorem 10.1]{HK05}, \cite[Chapter 16, Theorem 1]{HK2018}, for an alternative proof see \cite[Theorem 5.3]{GL2019})
An ergodic m.p.s. is a system of order $k$ iff it is isomorphic to a minimal $k$-step pronilsystem as m.p.s.
\end{thm}
\begin{rem}\label{rem:classical Z}
Let $(X,\mathcal{X},\mu,T)$ be an ergodic m.p.s.
 In the literature $(Z_k(X),\mathcal{Z}_k(X),\mu_k,T)$ is referred to as the \textbf{maximal measurable $k$-step pronilfactor} or as the \textit{maximal factor which is a system of order $k$} (see \cite[Chapter 9, Theorem 18]{HK2018}). By this it is meant that any measurable factor map $\phi:(X,\mathcal{X},\mu,T)\rightarrow (Y,\mathcal{Y},\nu,S)$ where $(Y,\mathcal{Y},\nu,S)$ is a minimal $k$-step pronilsystem, factors through the canonical $k$-th projection $\pi_k:(X,\mathcal{X},\mu,T)\rightarrow (Z_k(X),\mathcal{Z}_k(X),\mu_k,T)$, i.e., there exists a unique (up to measure zero) $\psi:(Z_k(X),\mathcal{Z}_k(X),\mu_k,T)\rightarrow  (Y,\mathcal{Y},\nu,S)$ such that $\phi=\psi\circ \pi_k$ a.s. In section \ref{uniqueness for Z and W} we establish the complementary property of \textit{universality} for $(Z_k(X),\mathcal{Z}_k(X),\mu_k,T)$.
\end{rem}

\begin{rem}
In \cite[Corollary 2.2]{HKM2014} a criterion for an ergodic m.p.s. $(X,\mathcal{X},\mu,T)$ to have $Z_k(X)=Z_1(X)$ for all $k\geq 1$ is given. Indeed this is the case for ergodic systems whose spectrum does not admit a
Lebesgue component with infinite multiplicity. In particular this holds true for weakly mixing systems, systems with singular maximal spectral type and systems with finite spectral multiplicity. 
\end{rem}
\subsection{Maximal topological  pronilfactors.}\label{subsec:classical top max}

Recall the Definition of $\mathcal{HK}^{k}(T)$ and $\mathcal{F}^k(T)$ (Definition \ref{def:HK cube group and face group}).
\begin{defn}
\label{defn of cubesystem} Let $(X,T)$ be a minimal t.d.s. Define the \textbf{induced $(k+1)$-th dynamical cubespace} by: $$\C_{\T}^{k+1}(X)=\overline{\{gx^{[k+1]}|\,g\in\mathcal{HK}^{k+1}(T)\}}.$$
\end{defn}

\begin{defn}(\cite[Definition 3.2]{HKM10})\label{defn for RP}
Let $(X,T)$ be a topological dynamical system and $k\geq 1$. The points $x,y\in X$ are said to be \strong{regionally proximal of order $k$}, denoted $(x,y)\in \RP^{[k]}(X)$, if there are sequences of elements $f_i\in \mathcal{F}^k(T)$, $x_i,y_i\in X$,  $a_*\in X_{*}^{[k]}$ such that 
$$\lim_{i\rightarrow \infty}(f_i x_i^{[k]},f_i y_i^{[k]})=(x,a_*,y,a_*).$$
\end{defn}

\begin{thm}\label{thm:rpk is eq rel}(\cite[Theorem 3.5]{SY12}\footnote{This theorem was generalized to arbitrary minimal group actions in \cite[Theorem 3.8]{GGY2018}.})
Let $(X,T)$ be a minimal t.d.s. and $k\geq 1$. Then $\RP^{[k]}(X)$ is a closed $T$-invariant equivalence relation.
\end{thm}

\begin{defn}
 A t.d.s. 
$(X,T)$ is called \strong{a (topological) system of order $k$} if  $\RP^{[k]}(X)=\{(x,x)\,|\, x\in X\}$.
\end{defn}

\begin{thm}(\cite[Theorem 1.2]{HKM10},  for an alternative proof see \cite[Theorem 1.30]{GMVIII})
A minimal t.d.s. is a topological system of order $k$ iff it is isomorphic to a  minimal $k$-step pronilsystem as t.d.s.
\end{thm}

Theorem \ref{thm:rpk is eq rel} allows us to give the following definition.

\begin{defn}\label{def:max nilfactor}
Let $(X,T)$ be a minimal t.d.s. Define the \textbf{maximal $k$-step nilfactor} by $W_k(X)=X/\RP^{[k]}(X)$. Denote the associated map $\pitk:X\rightarrow W_{k}(X)$ as the \textbf{(topological) canonical $k$-th projection}.
\end{defn}

\begin{rem}\label{rem:classical top max} 
The terminology of Definition \ref{def:max nilfactor} is justified by the following property: Any topological factor map $\phi:(X,T)\rightarrow (Y,T)$ where $(Y,T)$ is a system of order $k$,  factors through the canonical $k$-th projection $\pitk:(X,T)\rightarrow (W_k(X),T)$, i.e., there exists a unique $\psi:(W_k(X),T)\rightarrow  (Y,T)$ such that $\phi=\psi\circ \pitk$ (\cite[Proposition 4.5]{HKM10}). In section \ref{uniqueness for Z and W} we establish the complementary property of \textit{universality} for $(W_k(X),T)$.
\end{rem}
\begin{defn}(\cite[Definition 3.1]{GL2019})
A t.d.s. $(X,T)$ is called \textbf{$k$-cube uniquely ergodic} if 
$(\C_{\T}^{k}(X),\mathcal{HK}^{k}(T))$
is uniquely ergodic.
\end{defn}

\subsection{\cfnk systems.}\label{subsec: def:continuously to k-nil}

\begin{defn}\label{def:continuously to k-nil}
 For $k\geq 0$, we say $(X,T)$ is  a \strong{CF-Nil($k$)} system if $(X,T)$ is strictly ergodic and $(Z_k(X),\mathcal{Z}_k(X),\mu_k,T)$ is isomorphic to $(W_k(X),\omega_k,T)$ as m.p.s.where $\mu_k$ and $\omega_k$ are the images of the unique invariant measure of $(X,T)$ under the measurable, respectably topological canonical $k$-th projections.
\end{defn}
\begin{rem}\label{rem:0-nil}
 By convention  $Z_0(X)=W_0(X)=\{\bullet\}$. Thus every strictly ergodic $(X,T)$ is CF-Nil($0$).
\end{rem}
\noindent
The term "$(X,\mu, T)$ is CF-Nil($k$)" is an abbreviation of 
\vspace{0.25 cm}
\begin{center}
"$(X,\mu, T)$ \textbf{C}ontinuously \textbf{F}actors  on a $\mathbf{k}$-step pro\textbf{Nil}system which is isomorphic to $(Z_k(X),\mathcal{Z}_k(X),\mu_k,T)$ as m.p.s."
\end{center}
\vspace{0.25 cm}
Indeed if  $(W_k(X),\omega_k,T)$ is isomorphic to $(Z_k(X),\mathcal{Z}_k(X),\mu_k,T)$ as m.p.s. then obviously this condition holds. The reverse implication is given by the following proposition which has been (implicitly) used several times in the literature (\cite{HK09,HKM2014,HSY2019pointwise}). Its proof is given at the end of Subsection \ref{subsec:Universality}.

\begin{prop}\label{pronilfactor maximal}
Let $(X,T)$ be a strictly ergodic t.d.s. which topologically factors on a (minimal) $k$-step pronilsystem $(\hat{Z}_k,T)$ with the unique ergodic measure $\gamma_k$. If $(Z_k(X),\mathcal{Z}_k(X),\mu_k,T)$ is isomorphic to $(\hat{Z}_k,\gamma_k,T)$ as m.p.s., then $(\hat{Z}_k,T)$ and $(W_k(X),T)$ are isomorphic as t.d.s. In particular $(X,\mu, T)$ is \cfnkp
\end{prop}

Theorem \ref{thm:equiv cube u.e} allows us to give a remarkable simple proof of the following Theorem.

\begin{thm}\label{factor and k kronecker}
Let $(X,T)$ be a \cfnk system. The following holds:
\begin{enumerate}
    \item If $\pi:(X,T)\rightarrow (Y,T)$ is a topological factor map, then $(Y,T)$ is a \cfnk system.
    \item
    $(X,T)$ is a \cfni system for $0\leq i\leq k$. 
\end{enumerate}

\end{thm}

\begin{proof}
To prove (1) we note $(Y,T)$ is minimal being a factor of a minimal system and $(\C_{\T}^{k+1}(Y),\mathcal{HK}^{k+1}(T))$ is uniquely ergodic being a factor of $(\C_{\T}^{k+1}(X),\mathcal{HK}^{k+1}(T))$ under the natural topological factor map induced from  $\pi:(X,T)\rightarrow (Y,T)$ (see Proposition \ref{prop:factor of ue}). By Theorem \ref{thm:equiv cube u.e} this implies $(Y,T)$ is a \cfnk system.

\sloppy
Similarly, to prove (2), we consider 
$(\C_{\T}^{k+1}(X),\mathcal{HK}^{k+1}(T))\rightarrow (\C_{\T}^{i+1}(X),\mathcal{HK}^{i+1}(T))$ given by $$(c_{v_1,\ldots,v_{k+1}})_{(v_1,\ldots,v_{k+1})\in \{0,1\}^{k+1}}\mapsto (c_{v_1,\ldots,v_{i+1},0,\ldots, 0})_{(v_1,\ldots,v_{i+1})\in \{0,1\}^{i+1}}$$

\end{proof}

\subsection{A \cfnk topological model.} \label{subsec:Weiss thm}

Recall the definitions of Subsection \ref{subsec:topological model}. In \cite[Theorem 2]{W85} Benjamin Weiss proved the following theorem:

\begin{thm}(Weiss)\label{thm: Weiss thm}
Let  $(Z,\nu,S)$ be a strictly ergodic t.d.s. and $(X,\mathcal{X},\mu,S)$ an ergodic m.p.s. such that there exists a measurable factor $\pi:(X,\mathcal{X},\mu,T)\rightarrow (Z,\mathcal{Z},\nu,S)$. Then $\pi$ has a topological model $\hat{\pi}:(\hat{X},\hat{T})\rightarrow (Z,S)$ where $(\hat{X},\hat{T})$ is strictly ergodic.
\end{thm}

The following theorem is already implicit in \cite{HSY2019pointwise}.

\begin{thm}\label{thm:continuously to k-nil realization}
Let $k\in \Z$. Every ergodic system $(X,\mathcal{X},\mu,T)$ has a topological model $(\hat{X},\hat{T})$ such that $(\hat{X},\hat{T})$ is \cfnkp
\end{thm}
\begin{proof}

By Theorem \ref{thm:meas order k},   $(Z_k(X),\mathcal{Z}_k(X),\mu_k,T)$ is measurably isomorphic to a strictly ergodic inverse limit of $k$-step nilsystems $(\hat{Z}_k,\hat{T})$. By Theorem \ref{thm: Weiss thm}, $(X,\mathcal{X},\mu,T)$ admits a strictly ergodic topological model $(\hat{X},\hat{T})$ such that there exists a topological factor map $(\hat{X},\hat{T})\rightarrow (\hat{Z}_k,\hat{T})$ which is a topological model of $(X,\mathcal{X},\mu,T)\rightarrow(Z_k(X),\mathcal{Z}_k(X),\mu_k,T)$.
By Proposition \ref{pronilfactor maximal}, $(\hat{X},\hat{T})$ is \cfnkp
\end{proof}
\begin{rem}
 
One can easily construct a strictly ergodic system which is not \cfnkp\, Let $(X,\mathcal{X},\mu,T)$ be an irrational rotation on the circle. By \cite{Lehrer87}, there exists a topologically mixing and strictly ergodic model $(\hat{X},\hat{\mu},T)$ of $(X,\mu,T)$. As $X$ is an irrational rotation, $Z_1(\hat{X})=\hat{X}$ and therefore for all $k\geq 1$, $Z_k(\hat{X})=\hat{X}$. As $\hat{X}$ is  topologically mixing, it is  topologically weakly mixing and therefore for all $k\geq 1$, $W_k(\hat{X})=\{\bullet\}$ (\cite[Theorem 3.13(1)]{SY12}). It follows for all $k\geq 1$ one has that $(W_k(\hat{X}),T)$ is not isomorphic to $(Z_k(\hat{X}),\hat{\mu}_1,T)$ as m.p.s.



\end{rem}

\section{Coalescence and universality for maximal pronilfactors.}\label{uniqueness for Z and W}

\subsection{Coalescence}

In this section we establish Theorem \ref {thm:coalescense}, i.e., both \textit{topological coalescence} (introduced in \cite{auslander1963endomorphisms}) and \textit{measurable coalescence} (introduced in \cite{hahn1968some}) for minimal pronilsystems\footnote{The definitions of these concepts appear as part of the statements of Theorems \ref{thm: top coalescence} and \ref{thm: meas coalescence}  respectively.}. There is a vast literature dedicated to coalescence (see \cite{lemanczyk1992coalescence} and references within).  Coalescence plays an important role in the next subsection. 

\begin{thm} (Topological coalescence for minimal pronilsystems)\label{thm: top coalescence} Let $(Y,T)$ be a minimal $k$-step pronilsystem. Then $(Y,T)$  is topologically coalescent, i.e. if $\Phi:(Y,T)\rightarrow (Y,T)$ is a topological factor map, then $\Phi$ is a topological isomorphism.
\end{thm}
\begin{proof}
Recall that the Ellis semigroup is defined as $E=E(Y,T)=\overline{\{T^n:n\in \mathbb{Z}\}}$, where the closure is w.r.t. the product topology on $Y^Y$ (see \cite{E58} for more details). By a theorem of Donoso \cite[Theorem 1.1]{DonosoEnveloping}, $E(Y,T)$ is a $k$-step nilpotent group, i.e. for $E_1=E$, $E_{i+1}=[E_{i},E], i\geq 1$, one has that $E_{k+1}=\{\id\}$. As $\Phi$ is continuous, one has that $E$ and $\Phi$ commute, i.e. for any $g\in E$, $\Phi\circ g=g\circ \Phi$. For any $z\in Y$, we define the group $\mathcal{G}(Y,z)=\{\alpha \in E(Y,T), \alpha z=z\}$. Let $x,y\in Y$ such that $\Phi(x)=y$. If $u\in \mathcal{G}(Y,x)$, one always has that $uy=u(\Phi(x))=\Phi(ux)=\Phi(x)=y$, i.e. $u\in \mathcal{G}(Y,y)$. Thus $\mathcal{G}(Y,x)\subset \mathcal{G}(Y,y)$.

Assume that $\Phi$ is not one-to-one, then there exists $x_1\neq x_2 \in Y$ such that $\Phi(x_1)=\Phi(x_2)$. As $(Y,T)$ is minimal, there exists $p_1,p_2\in E(Y,T)$ such that $x_1=p_1x$, $x_2=p_2x$. Then $p_1y=\Phi(p_1x)=\Phi(x_1)=\Phi(x_2)=\Phi(p_2x)=p_2y$. Thus $p_1^{-1}p_2\in \mathcal{G}(Y,y)$. As $p_2x=x_2\neq x_1=p_1 x$,  we have $$p_1^{-1}p_2x\neq x,$$ which implies that $p_1^{-1}p_2\in \mathcal{G}(Y,y)\setminus  \mathcal{G}(Y,x)$.

Let $\beta_0=p_1^{-1}p_2$. As $(Y,T)$ is minimal, there exists $u\in E(Y,T)$ such that $ux=y$. Then $\mathcal{G}(Y,x)=u^{-1}\mathcal{G}(Y,y)u$. Let $\beta_1=(u^{-1}\beta_0^{-1} u)\beta_0$. As $\beta_0\in \mathcal{G}(Y,y)\setminus  \mathcal{G}(Y,x)$, one has that 
\begin{equation}\label{beta0 property}\beta_0\notin \mathcal{G}(Y,x),\beta_0\in \mathcal{G}(Y,y)\text{ and }(u^{-1}\beta_0^{-1} u)\in u^{-1}\mathcal{G}(Y,y)u=\mathcal{G}(Y,x)\subset \mathcal{G}(Y,y).
\end{equation}
Thus we can show that $\beta_1\in \mathcal{G}(Y,y)\setminus  \mathcal{G}(Y,x)$. Indeed, by \eqref{beta0 property} we know that $\beta_1=(u^{-1}\beta_0^{-1}u)\beta_0\in  \mathcal{G}(Y,y)$ as $ \mathcal{G}(Y,y)$ is a group. If $\beta_1\in \mathcal{G}(Y,x)$, then $\beta_0=(u^{-1}\beta_0^{-1}u)^{-1}\beta_1\in \mathcal{G}(Y,x)$, which constitutes a contradiction. Therefore $\beta_1\in \mathcal{G}(Y,y)\setminus  \mathcal{G}(Y,x)$  and $(u^{-1}\beta_1^{-1}u)\in  u^{-1}\mathcal{G}(Y,y)u=\mathcal{G}(Y,x)$.

Similarly, we define $\beta_{i+1}=(u^{-1}\beta_{i}^{-1}u)\beta_i$ for $i\geq 1$. By the same argument, one has that $\beta_{i+1}\in \mathcal{G}(Y,y)\setminus  \mathcal{G}(Y,x)$. But notice that $ \beta_{i}\in E_{i+1}$ and $E_{k+1}=\{\id\}$, therefore $\id=\beta_k\in \mathcal{G}(Y,y)\setminus  \mathcal{G}(Y,x)$. Contradiction.

Thus $\Phi$ is a one-to-one topological factor map, which implies it is a topological isomorphism.
\end{proof}

\begin{prop}\cite[Chapter 13, Proposition 15]{HK2018}\label{continuty for pronilsystem factor map}
Let $(Y,\nu,T)$, $(Y',\nu',T)$ be minimal (uniquely ergodic) $k$-step pronilsystems. Let $\pi:(Y,\nu,T)\rightarrow (Y',\nu',T)$ be a measurable factor map. Then there exists a topological factor map $\hat{\pi}:(Y,T)\rightarrow (Y',T)$ such that $\pi(y)=\hat{\pi}(y)$ for $\nu$-a.e. $y$.
\end{prop}

Combining Theorem \ref{thm: top coalescence} and Proposition \ref{continuty for pronilsystem factor map} we immediately have the following theorem. 

\begin{thm} (Measurable coalescence for minimal pronilsystems)\label{thm: meas coalescence} Let $(Y,\nu,T)$ be a minimal (uniquely ergodic) $k$-step pronilsystem. Then $(Y,\nu,T)$  is measurably coalescent, i.e. if $\pi:(Y,\nu,T)\rightarrow (Y,\nu,T)$ is a measurable factor map, then $\pi$ is a measurable  isomorphism (which equals a.s. a topological isomorphism).
\end{thm}
\begin{proof}
By Proposition \ref{continuty for pronilsystem factor map}, there exists a topological factor map $\hat{\pi}:(Y,\nu,T)\rightarrow (Y,\nu,T)$ such that $\pi(y)=\hat{\pi}(y)$ for $\nu$-a.e. $y\in Y$. By Theorem \ref{thm: top coalescence}, $\hat{\pi}$ is a topological isomorphism. As $\pi$ equals a.s. $\hat{\pi}$, one may find a $T$-invariant Borel set $Y_0\subset Y$ with $\nu(Y_0)=1$, $\pi_{|Y_0}=\hat{\pi}_{|Y_0} $. As $\hat{\pi}$ is one-to-one, $\pi_{|Y_0}^{-1}(\pi_{|Y_0}(Y_0))=Y_0$ and therefore $\nu(\pi_{|Y_0}(Y_0))=1 $. Thus 
$\pi_{|Y_0}:Y_0\rightarrow \hat{\pi}(Y_0)$ is a Borel measurable one-to-one map between two $T$-invariant sets of full measure, which implies that $\pi$ is a measurable  isomorphism. 
\end{proof}

\begin{cor}\label{cor:strong max}
\sloppy Let $(X,\mathcal{X},\mu,T)$ be an ergodic m.p.s. and $k\in \mathbb{N}$. Let $(Y,\mathcal{Y},\nu,S)$ be a minimal $k$-step pronilsystem isomorphic to $(Z_k(X),\mathcal{Z}_k(X),\mu_k,T)$. Let $\pi:(X,\mathcal{X},\mu,T)\rightarrow (Y,\mathcal{Y},\nu,S)$ be a factor map. 
 The following holds:
 \begin{enumerate}
     \item 
 There is a (topological) isomorphism $p~:~(Z_k(X),\mathcal{Z}_k(X),\mu_k,T)\rightarrow (Y,\mathcal{Y},\nu,S)$ such that $\pi=p\circ \pi_k$ a.s.
 \item
 For every measurable factor map $\phi:(X,\mathcal{X},\mu,T)\rightarrow (Y',\mathcal{Y}',\nu',S')$ where $(Y',\mathcal{Y}',\nu',S')$ is a minimal $k$-step pronilfactor, factors through $\pi$,  i.e., there exists a unique (up to measure zero) $\psi:(Y,\mathcal{Y},\nu,S)\rightarrow  (Y',\mathcal{Y}',\nu',S')$ such that $\phi=\psi\circ \pi$ a.s.
\end{enumerate}

$$
\xymatrix@R=0.7cm{
  & X\ar[dl]_{\pi_{k}}\ar[d]^{\pi}\ar[dr]^{\phi} &\\
  Z_k\ar[r]^{p}& \ar[l]Y\ar[r]^{\psi} &Y' 
  }
$$
\end{cor}
\begin{proof}
By the maximality of $\pi_k$ (see Subsection \ref{subsec:classical meas max}) there is a measurable factor map $p:(Z_k(X),\mathcal{Z}_k(X),\mu_k,T)\rightarrow (Y,\mathcal{Y},\nu,S)$ such that  $\pi=p\circ \pi_k$ a.s.
By assumption there is a measurable isomorphism $i:(Y,\mathcal{Y},\nu,S)\rightarrow(Z_k(X),\mathcal{Z}_k(X),\mu_k,T)$ (which equals a.s. a topological isomorphism). By Theorem \ref{thm: meas coalescence}, $i \circ p$ is a measurable isomorphism and therefore $p$ is a measurable isomorphism. This establishes (1). Thus $\pi$ inherits the maximality property of $\pi_k$. This establishes (2).
\end{proof}

\begin{rem}
 Bernard Host has pointed out to us that it is possible to prove Theorem \ref{thm:coalescense} using results from \cite[Chapter 13]{HK2018}.
\end{rem}

\subsection{Universality}\label{subsec:Universality}

\begin{defn}
Let $(X,\mu,T)$ be a strictly ergodic t.d.s. Denote by $\cktop$ the collection of (topological) isomorphism equivalence classes of topological $k$-step pronilfactors of $(X,T)$. Denote by $\ckme$ the collection of (measurable) isomorphism equivalence classes of measurable $k$-step pronilfactors of $(X,T)$. An (equivalence class of) t.d.s. $(M,T)\in \cktop$ is called $\cktop$-\textbf{universal}\footnote{This terminology is frequently used in the literature, see \cite{dV93,GutLi2013}.} if every $(N,S)\in\cktop$ is a topological factor of $(M,T)$. An (equivalence class of) m.p.s. $(M,\mathcal{M},\mu,T)\in \ckme$ is called $\ckme$-\textbf{universal} if every $(N,\mathcal{N},v,S)\in\ckme$ is a measurable factor of $(M,\mathcal{M},\mu,T)$.
\end{defn}

The following theorem establishes a complementary property to maximality as described in                                  Remark \ref{rem:classical Z} and Remark \ref{rem:classical top max}.

\begin{thm}
Let $(X,\mu,T)$ be a strictly ergodic t.d.s., then $(W_k(X),T)$ is the unique $\cktop$-universal topological $k$-step pronilfactor of $(X,T)$ and $(Z_k(X),\mathcal{Z}_k(X),\mu_k,T)$ is the unique $\ckme$-universal measurable $k$-step pronilfactor of $(X,T)$.
\end{thm}

\begin{proof}

By Remark \ref{rem:classical Z} $(Z_k(X),\mathcal{Z}_k(X),\mu_k,T)$ is a $\ckme$-universal measurable $k$-step pronilfactor of $(X,T)$. Assume  $(Z'_k(X),\mathcal{Z}'_k(X),\mu'_k,T)$ is another $\ckme$-universal measurable $k$-step pronilfactor of $(X,T)$. By universality one has measurable factor maps $Z'_k(X)\rightarrow \mathcal{Z}'_k(X)$ and $Z_k(X)\rightarrow \mathcal{Z}'_k(X)$. 
By  Theorem \ref{thm: meas coalescence},  $Z_k(X)$ and $\mathcal{Z}'_k(X)$ are isomorphic.

By Remark \ref{rem:classical top max} $(W_k(X),T)$ is a $\cktop$-universal topological $k$-step pronilfactor of $(X,T)$. By Theorem \ref{thm: top coalescence} it is unique.

\end{proof}

\begin{proof}[Proof of Proposition \ref{pronilfactor maximal}]
By Remark \ref{rem:classical top max}, one can find a topological factor map $q:(W_k(X),T)\rightarrow (\hat{Z}_k,T)$. Let $\omega_k$ be the unique ergodic measure of $(W_k(X),T)$. By Remark \ref{rem:classical Z}, one can find a measurable factor map $\psi:(\hat{Z}_k,\gamma_k, T)\rightarrow (W_k(X),\omega_k,T)$. $$
\xymatrix@R=0.7cm{
\hat{Z}_k\ar@{.>}[r]^{\psi} & W_k  \ar@/^/[l]^{q}  }
$$
By Proposition \ref{continuty for pronilsystem factor map}, there exists a topological factor map $\hat{\psi}:(\hat{Z}_k,\gamma_k, T)\rightarrow (W_k(X),\omega_k,T)$ such that $\hat{\psi}=\psi$ a.s. In particular, $\hat{\psi}\circ q:(W_k(X),\omega_k,T)\rightarrow (W_k(X),\omega_k,T)$ is a topological factor map. By Theorem \ref{thm: top coalescence}, $\hat{\psi}\circ q$ is a topological isomorphism. Thus $q$ is a topological isomorphism. As $(\hat{Z}_k,T)$ and $(W_k,T)$ are uniquely ergodic, $q$ is also a measurable isomorphism. In particular $(W_k(X),\mathcal{W}_k(X),\omega_k,T)$ and $(Z_k(X),\mathcal{Z}_k(X),\mu_k,T)$ are isomorphic as m.p.s. and $(X,\mu,T)$ is \cfnkp
\end{proof}

\section{Cubespace characterization of \cfnkp}\label{sec:cubespace}

In this section, we prove Theorem \ref{thm:equiv cube u.e}. We need some lemmas.
\begin{lem}\cite[Lemma 5.6]{HKM10}\label{prop:full support on Ck1}
Let $(X,T)$ be a minimal topological dynamical system and $\mu$ be an invariant ergodic measure on $X$. Then the measure $\mu^{[k]}$ is supported on $\C_{\T}^k(X)$  for any $k\geq 1$.
\end{lem}

\begin{proof}
The Lemma is proven in \cite[Lemma 5.6]{HKM10} with the help of the so called $L^2$-convergence of cubical averages theorem \cite[Theorem 1.2]{HK05}. This is a deep theorem with a highly non-trivial proof. We note that we are able to give a direct proof of this Lemma which we hope to publish elsewhere. 
\end{proof}

\begin{defn}\label{Folner sequence}
Let $G$ be a countable amenable group. A \textbf{Følner sequence} $\{F_{N}\}_{N\in \mathbb{N}}$ is a sequence of finite subsets of $G$ such that for any $g\in G$, $\lim_{n\rightarrow \infty}|gF_{N}\cap F_{N}|/|F_{N}|=1$. 

\end{defn}

\begin{thm}\label{pointwise ergodic theorem for amenable}(Lindenstrauss)
Let $G$ be an amenable group acting on a measure space $(X,\mathcal{X},\mu)$ by measure preserving transformations.
Let $\mathcal{I}_G$ be the $G$-invariant
$\sigma$-algebra of $(X,\mathcal{X},\mu)$.
There is a Følner sequence $\{F_{N}\}_{N\in \mathbb{N}}$ such that for any $f\in L^\infty(\mu)$, for $\mu$-a.e. $x\in X$, 
$${\displaystyle \lim_{N\rightarrow\infty}\frac{1}{|F_{N}|}\sum_{g\in F_{N}}}f(gx)=\mathbb{E}(f|\mathcal{I}_G)(x), 
$$
In particular, if the $G$ action is ergodic, for $\mu$-a.e. $x\in X$, 
$${\displaystyle \lim_{N\rightarrow\infty}\frac{1}{|F_{N}|}\sum_{g\in F_{N}}}f(gx)=\int f(x)d\mu\text{ a.e.}$$
\end{thm}
\begin{proof}
The theorem follows from \cite[Theorem 1.2]{L01} and \cite[Proposition 1.4]{L01}. In  \cite[Theorem 1.2]{L01} the statement reads 
\begin{equation}\label{eq:Lin}
 \lim_{N\rightarrow\infty}\frac{1}{|F_{N}|}\sum_{g\in F_{N}}f(gx)=\overline{f}(x)\text{ a.e.}
\end{equation}
for some $G$-invariant $\overline{f}\in L^\infty(\mu)$. 

Note that if we replace $f$ by $\mathbb{E}(f|\mathcal{I}_G)$ in \eqref{eq:Lin}, we have trivially as $\mathbb{E}(f|\mathcal{I}_G)$ is $G$-invariant: 
$$
\mathbb{E}(f|\mathcal{I}_G)(x)=\lim_{N\rightarrow\infty}\frac{1}{|F_{N}|}\sum_{g\in F_{N}}\mathbb{E}(f|\mathcal{I}_G)(gx)
$$
Using the Lebesgue dominated convergence theorem for conditional expectation\footnote{It follows easily from applying the Lebesgue dominated convergence theorem in Equation \eqref{eq:cond exp}.}  one has:
 $$
 \mathbb{E}(f|\mathcal{I}_G)(x)=\lim_{N\rightarrow\infty}\mathbb{E}(\frac{1}{|F_{N}|}\sum_{g\in F_{N}}f(g\cdot)|\mathcal{I}_G)(x)=\mathbb{E}(\overline{f}|\mathcal{I}_G)(x)=\overline{f}(x)\text{ a.e.}
 $$
Thus $\overline{f}(x)=\mathbb{E}(f|\mathcal{I}_G)(x)$, which gives the statement above.

\end{proof}

\begin{proof}[Proof of Theorem \ref{thm:equiv cube u.e}]
(I) $\Rightarrow$ (II): This follows from the proof in \cite[Section 4.4.3]{HSY2017}, where it is shown  that if one has a commutative diagram of the following form:
\[
\begin{CD}
(X,\mathcal{X},\mu,T) @>{\phi}>> (\h{X},T)\\
@V{\pi_k}VV      @VV{\h{\pi}_k}V\\
(Z_k(X),\mathcal{Z}_k(X),\mu_k,T) @>{\id }>> (Z_k(X),T),
\end{CD}
\]
 then $(C^{k+1}_T(\hat{X}),\mathcal{HK}^{k+1}(T)) $ is uniquely ergodic. Here $(X,\mathcal{X},\mu,T)$ is an ergodic system, $(\hat{X},T)$ is strictly ergodic, $\phi$ is a measurable isomorphism w.r.t. the uniquely ergodic measure of $(\hat{X},T)$ and $\hat{\pi}_k$ is a topological factor map. Indeed, it is easy to obtain such a diagram for a \cfnk system using Proposition \ref{pronilfactor maximal}.

(II) $\Rightarrow$ (I): We assume that $(\C_{\T}^{k+1}(X),\mathcal{HK}^{k+1}(T))$ is uniquely ergodic. By Lemma \ref{prop:full support on Ck1}, the unique invariant measure is $\mu^{[k+1]}$. As $(X,T)$ is a topological factor of $(\C_{\T}^{k+1}(X),\mathcal{HK}^{k+1}(T))$ w.r.t. the projection to the first coordinate, $(X,T)$ is uniquely ergodic.

Let $p_k:(X,T)\rightarrow (W_k(X),T)$ be the topological canonical $k$-th projection. By Proposition \ref{prop:factor of ue}, as $(X,T)$ is uniquely ergodic so is $(W_k(X),T)$. Let us denote by $\omega_k$ the unique invariant measure of $(W_k(X),T)$. Obviously $(p_k)_*\mu=\omega_k$. Thus  $p_k:(X,\mu,T)\rightarrow (W_k(X),\omega_k,T)$ is a measurable factor map.
Let $\mathcal{W}_k$ be the $\sigma$-algebra corresponding to the map $p_k$. Let $\mathcal{Z}_k$ be the $\sigma$-algebra corresponding to the measurable canonical $k$-th projection $\pi_k:(X,\mu,T)\rightarrow (Z_k(X),\mathcal{Z}_k(X),\mu_k,T)$. We will show that $\mathcal{W}_k=\mathcal{Z}_k$, which implies that $(W_k(X),\omega_k,T)$ is isomorphic to $(Z_k(X),\mathcal{Z}_k(X),\mu_k,T)$ as m.p.s.
The map $p_k:(X,T)\rightarrow (W_k(X),T)$ induces a factor map $$(\C_{\T}^{k+1}(X),\mathcal{HK}^{k+1}(T))\rightarrow(\C_{\T}^{k+1}(W_k(X)),\mathcal{HK}^{k+1}(T)).$$ By Proposition \ref{prop:factor of ue}, as $(\C_{\T}^{k+1}(X),\mathcal{HK}^{k+1}(T))$ is uniquely ergodic so is
$(\C_{\T}^{k+1}(W_k(X)),\mathcal{HK}^{k+1}(T))$.
By Lemma \ref{prop:full support on Ck1} the unique invariant measure on $(\C_{\T}^{k+1}(W_k(X)),\mathcal{HK}^{k+1}(T))$ is $\omega_k^{[k+1]}$.
Let $\gamma_{k+1}$ be the \textit{conditional product measure relative to $(W_k(X)^{[k+1]},\omega_k^{[k+1]})$} on $X^{[k+1]}$ (\cite[Definition 9.1]{F77}). This is the unique measure on $X^{[k+1]}$ such that for all $f_v\in L^\infty(X,\mu)$, $v\in\{0,1\}^{k+1}$ (\cite[Lemma 9.1]{F77}):
\begin{multline}\label{definition of gamma}\int_{X^{[k+1]}} \prod_{v\in \{0,1\}^{k+1}}f_v(c_v)d\gamma_{k+1} (c)=\\\int_{W_k(X)^{[k+1]}} \prod_{v\in \{0,1\}^{k+1}}
\mathbb{E}(f_v|W_k(X))(c_v)d\omega_k^{[k+1]}(c).
\end{multline}
As $\mathbb{E}(\cdot|W_k(X))$ commutes with $T$ and $\omega_k^{[k+1]}$ is $\mathcal{HK}^{k+1}(T)$-invariant, one has that $\gamma_{k+1}$ is $\mathcal{HK}^{k+1}(T)$-invariant.
It is natural to introduce the measure $\gamma_{k+1}$ as by \cite[Chapter 9, Theorem 14]{HK2018}, $\mu^{[k+1]}$ is the  conditional product measure relative to $\mu_{k}^{[k+1]}$. Thus if $\mu_{k}=\omega_k$ then $\gamma_{k+1}=\mu^{[k+1]}$. It turns out one can reverse the direction of implications.
Indeed we claim that $\gamma_{k+1}(\C_{\T}^{k+1}(X))=1$.
Assuming this claim and recalling the assumption that  $(\C_{\T}^{k+1}(X),\mathcal{HK}^{k+1}(T))$ is uniquely ergodic, one has by Lemma \ref{prop:full support on Ck1} that $\gamma_{k+1}=\mu^{[k+1]}$.
With the end goal of showing $\mathcal{Z}_{k}= \mathcal{W}_k$ we start by showing $\mathcal{Z}_{k}\subset \mathcal{W}_k$. It is enough to show $L^\infty(\mu)\cap L^2(\mathcal{W}_k)^\perp\subset L^\infty(\mu)\cap L^2(\mathcal{Z}_{k})^\perp$.
To this end we will show that for any function $f\in L^\infty(\mu)$ such that $\mathbb{E}(f|\mathcal{W}_k)=0$, it holds that $\mathbb{E}(f|\mathcal{Z}_{k})=0$. By Definition \ref{def:The-seminorms}, as $\gamma_{k+1}=\mu^{[k+1]}$,
\begin{multline*}
|||f|||_{k+1}^{2^{k+1}}=\int\prod_{v\in\{0,1\}^{k+1}}\mathcal{C}^{|v|}f(c_v)d\gamma_{k+1}(c)=\\
\int \prod_{v\in \{0,1\}^{k+1}}\mathbb{E}(\mathcal{C}^{|v|}f|W_k(X))(c_v)d\omega_k^{[k+1]}(c).
\end{multline*}
As $\mathbb{E}(f|\mathcal{W}_k)\equiv0$, it holds that $\mathbb{E}(\mathcal{C}^{|v|}f|W_k(X))\equiv0$ for any $v\in \{0,1\}^{k+1}$. Therefore  $|||f|||_{k+1}=0$. This implies by Lemma \ref{lem: seminorm and k factor} that
$\mathbb{E}(f|\mathcal{Z}_{k})=0$ as desired. 
By Remark \ref{rem:classical Z}, $Z_k(X)$ is the maximal measurable $k$-step pronilfactor of $(X,\mu,T)$. As $(W_k(X),\omega_k,T)$ is a $k$-step pronilfactor of $(X,T)$, one has that $\mathcal{W}_k\subset \mathcal{Z}_{k}$. Thus $\mathcal{W}_k=\mathcal{Z}_k$, which implies that $(W_k(X),\omega_k,T)$ is isomorphic to $(Z_{k}(X),\mathcal{Z}_k(X),\mu_{k},T)$ as m.p.s.

As a final step, we will now show that  $\gamma_{k+1} (\C_{\T}^{k+1}(X))=1$. Let $f_v\in L^\infty(X,\mu)$, $v\in \{0,1\}^{k+1}$ and set $H_0=\prod_{v\in \{0\}\times \{0,1\}^k}f_v$ and $H_1=\prod_{v\in \{1\}\times \{0,1\}^k}f_v$ as well as $\hat{H}_0=\prod_{v\in \{0\}\times \{0,1\}^k}\mathbb{E}(f_v|W_k(X))$, $ \hat{H}_1=\prod_{v\in \{1\}\times \{0,1\}^k}\mathbb{E}(f_v|W_k(X))$.
By Equation \eqref{definition of gamma}, we have 
\begin{equation}\label{eq:first}
\int_{X^{[k+1]}} H_0(c)H_1(c')d\gamma_{k+1}(c,c') =\int_{W_k(X)^{[k+1]}} \hat{H}_0(c)\hat{H}_1(c')d\omega_k^{[k+1]}(c,c').
\end{equation}
By Equation \eqref{def muk 2} in Definition \ref{def: muk definition},
\begin{equation}\label{eq:second}
\int_{W_k(X)^{[k+1]}} \hat{H}_0(c)\hat{H}_1(c')d\omega_k^{[k+1]}(c,c')=\int_{W_k(X)^{[k]}} \mathbb{E}(\hat{H}_0|\mathcal{I}_{T^{[k]}})(c)\hat{H}_1(c)d\omega_k^{[k]}(c).    
\end{equation}
By Birkhoff's ergodic theorem (see also Theorem \ref{pointwise ergodic theorem for amenable}), one has that 
\begin{equation}\label{eq:3 lines}\begin{array}{ll}\int_{W_k(X)^{[k]}} \mathbb{E}(\hat{H}_0|\mathcal{I}_{T^{[k]}})(c)\hat{H}_1(c)d\omega_k^{[k]}(c)\\
{\displaystyle =\int\lim_{N\rightarrow \infty}\frac{1}{N}\sum_{n=0}^{N-1}\hat{H}_0((T^{[k]})^nc)\hat{H}_1(c)d\omega_k^{[k]}(c)}\\
={\displaystyle \lim_{N\rightarrow \infty}\frac{1}{N}\sum_{n=0}^{N-1}\int\hat{H}_0((T^{[k]})^nc)\hat{H}_1(c)d\omega_k^{[k]}(c)},
\end{array}
\end{equation}
here we used the Lebesgue dominated convergence theorem. 

Abusing notation one may consider $\hat{H}_0$ and $\hat{H}_1$ as defined on $X^{[k]}$ (see Subsection \ref{Subsec: Con exp}). As $p_k:(X,\mu,T)\rightarrow (W_k(X),\omega_k,T)$ is a measurable factor map, one has 
$$\int\hat{H}_0((T^{[k]})^nc)\hat{H}_1(c)d\omega_k^{[k]}(c)=\int\hat{H}_0((T^{[k]})^nc)\hat{H}_1(c)d\mu^{[k]}(c). $$
 As $(C_T^k(X),\mathcal{HK}^{k}(T))$  is a topological factor of $(\C_{\T}^{k+1}(X),\mathcal{HK}^{k+1}(T))$ w.r.t. the  ``lower'' $2^k$ coordinates, $(C_T^k(X),\mathcal{HK}^{k}(T))$  is uniquely ergodic. By Lemma \ref{prop:full support on Ck1}, the unique ergodic measure is $\mu^{[k]}$. 
 By Theorem \ref{pointwise ergodic theorem for amenable} applied to $(C_T^k(X),\mu^{[k]},\mathcal{HK}^{k}(T))$, there is  a Følner sequence $\{F_M\subset \mathcal{HK}^k(T)\}_{M\in \mathbb{N}}$ such that 
 \begin{equation}\label{eq:first sum}
 \int\hat{H}_0\big((T^{[k]})^nc\big)\hat{H}_1(c)d\mu^{[k]}(c)=\lim_{M\rightarrow \infty}\frac{1}{|F_M|}\sum_{h\in F_M}\hat{H}_0\big((T^{[k]})^nhs\big)\hat{H}_1(hs)    
 \end{equation}
 for $\mu^{[k]}$-a.e. $s\in C_T^k(X)$.
Thus from  Equations \eqref{eq:first},  \eqref{eq:second}, \eqref{eq:3 lines} and \eqref{eq:first sum}, it holds for arbitrary $f_v\in L^\infty(X,\mu)$, $v\in \{0,1\}^{k+1}$, $H_0=\prod_{v\in \{0\}\times \{0,1\}^k}f_v$ and $H_1=\prod_{v\in \{1\}\times \{0,1\}^k}f_v$, for $\mu^{[k]}$-a.e. $s\in C_T^k(X)$, 
\begin{multline}
\label{ergodic convergence for Ck}
    \int_{X^{[k+1]}} H_0(c)H_1(c')d\gamma_{k+1}(c,c')=\\
    \lim_{N\rightarrow \infty}\frac{1}{N}\sum_{n=0}^{N-1}\lim_{M\rightarrow \infty}\frac{1}{|F_M|}\sum_{h\in F_M}\hat{H}_0\big((T^{[k]})^nhs\big)\hat{H}_1(hs)
\end{multline}
Let $R\in C(X^{[k+1]},\mathbb{R})$ be a continuous function. We claim
 for $\mu^{[k]}$-a.e. $s\in \C_{\T}^k(X)$,
\begin{equation}\label{uniformly converges-1}\int R(c)d\gamma_{k+1}(c)=\lim_{N\rightarrow \infty}\frac{1}{N}\sum_{n=0}^{N-1}\lim_{M\rightarrow \infty}\frac{1}{|F_M|}\sum_{h\in F_M} R\big((T^{[k]})^nhs,hs\big)
\end{equation}
Notice that it follows from Definitions \ref{def:HK cube group and face group} and 
\ref{defn of cubesystem} that if $s\in C_T^k(X)$, then $((T^{[k]})^nhs,hs)\in C_T^{k+1}(X)$ for arbitrary $h\in \mathcal{HK}^k(T)$ and $n\in \mathbb{Z}^{+}$ (see also \cite[Subsection A.2]{GGY2018}). Thus using Equation \eqref{uniformly converges-1} with functions $R_\delta\in C(X^{[k+1]},[0,1])$ such that $R_\delta|_{\C_{\T}^{k+1}(X)}\equiv 1$ and $R|_{X^{[k+1]}\setminus B_{\delta}(\C_{\T}^{k+1}(X))}\equiv 0$,  (taking $\delta$ to zero) one obtains: 
 $$\gamma_{k+1}(C_T^{k+1}(X))=1.$$
We now prove \eqref{uniformly converges-1}. For $d\in \mathbb{N}$, let $H_d^{(i)}$ be functions of the form $\prod_{v\in \{0,1\}^{k+1}} h^{(i)}_v$, $i\in I_d$ for some finite set $I_d$, such that $|R(z)-\sum_{i\in I_d}H_d^{(i)}(z)|<\frac{1}{2d}$ for all $z\in \C_{\T}^{k+1}(X)$. 
Denote by $C(R)=\int R(c)d\gamma_{k+1}(c)$ the (LHS) of \eqref{uniformly converges-1}. Denote by 
 $D(R)(z)$ be the (RHS) of Equation \eqref{uniformly converges-1}.  By Equation \eqref{ergodic convergence for Ck},
$C(H_d^{(i)})=D(H_d^{(i)})(z)$ for $\mu^{[k]}$-a.e. $z\in \C_{\T}^k(X)$. Note that $|C(R)-\sum_{i\in I_d}C(H_d^{(i)})|<\frac{1}{2d}$ and $|D(R)(y)-\sum_{i\in I_d}D(H_d^{(i)})(y)|<\frac{1}{2d}$ for all $y\in \C_{\T}^k(X)$. Thus for any $d$, $E_d:=\{y\in \C_{\T}^{k}(X):|C(R)(y)-D(R)(y)|<\frac{1}{d}\}$ has full $\mu^{[k]}$ measure. Let $E=\bigcap_{d\in \mathbb{N}}E_d$, then $\mu^{[k]}(E)=1$ and for any $y\in E$, Equation \eqref{uniformly converges-1} holds.
\end{proof}
The following remark may be of interest:

\begin{rem}
 In \cite[Section 6]{GHSY2019} an example is given showing there exists a strictly ergodic \textit{distal} system which is \textit{not} \cfnone.

\end{rem}

\section{A topological Wiener-Wintner theorem.}\label{sec TWWT}

In this section, we prove Theorem \ref{TWWT main}.

\begin{defn}\label{def:generic}
Let $(X,T)$ be a t.d.s. and $\mu\in \PM(X)$. A point $x\in X$ is \textbf{generic} (for $\mu$) if for all $f\in C(X)$
$$\lim_{N\rightarrow \infty}\frac{1}{N}\sum_{n=0}^Nf(T^nx)=\int f d\mu$$
\end{defn}

\begin{lem} \label{lem:genericity}
Let $(X,T)$ be a t.d.s. and $x_0\in X$. Assume that for all $f\in C(X)$, there exists $c_f\in \R$, a constant depending on $f$, so that :
$$\lim_{N\rightarrow \infty}\frac{1}{N}\sum_{n=0}^N f(T^nx_0)=c_f$$
Then $x_0$ is generic for some $\mu\in \PM(X)$.
\end{lem}

\begin{proof}
Define the functional $\phi:C(X)\rightarrow \R$ by $\phi(f)=c_f$. It is easy 
to see that $\phi$ is a bounded linear positive functional of supremum norm $1$. By the Riesz representation theorem $c_f=\int f d\mu$ for some Borel probability measure $\mu$ on 
$X$ (\cite[Theorem 2.14]{rudin2006real}). As $c_f=c_{Tf}$ for all $f\in C(X)$, it follows that $\mu\in \PM(X)$. Thus $x_0$ is generic by Definition \ref{def:generic}. 
\end{proof}

\begin{thm}(\cite[Theorem 4.10]{G03})\label{thm:genericity}
Let $(X,T)$ be a minimal t.d.s., then  $(X,T)$ is uniquely ergodic iff every  $x\in X$ is generic for some $\mu\in \PM(X)$ (depending on $x$).
\end{thm}

\begin{lem}\label{lem:supp gen}
Let $(X,T)$ be a t.d.s. and $\mu\in \PM(X)$. If a point $x\in X$ is generic for $\mu$, then $\mu$ is supported on $\OC(x)$.
\end{lem}

\begin{proof}
Let $f$ be a non-negative function supported outside $\OC(x)$. Then $\int f d\mu=\lim_{N\rightarrow \infty}\frac{1}{N}\sum_{n=1}^Nf(T^nx)=0$. Q.E.D.
\end{proof}

\begin{proof}[Proof of Theorem \ref{TWWT main}]

$(I)\Rightarrow (II)$. It follows from 
\cite[Theorem 2.19 and Proposition 7.1]{HK09}.

We will show $(II)\Rightarrow (I)$ inductively. For $k=0$ note that Condition $(II)$ with the constant nilsequence $a(n)\equiv 1$ implies that  for a fixed arbitrary $x\in X$ and every $f\in C(X)$, $\lim_{N\rightarrow \infty}\frac{1}{N}\sum_{n=1}^Na(n)f(T^n x)=\lim_{N\rightarrow \infty}\frac{1}{N}\sum_{n=1}^N f(T^n x)$ exists. From Lemma \ref{lem:genericity}, $x\in X$ is generic for some $\mu_x\in P_T(X)$. By Theorem \ref{thm:genericity}, $(X,T)$  is uniquely ergodic. By assumption  $(X,T)$ is minimal and thus $(X,T)$ is a \cfnk system.

Assume the $(II)\Rightarrow (I)$ holds for $k-1$.  We will now show  $\sim (I)\Rightarrow \,\,\, \sim (II)$ for $k$. Thus we assume that $(X,T)$ is not \cfnkp
\,\,If $(X,T)$ is not \cfnkm then the result follows from the inductive assumption. Thus we may assume $(X,T)$ is \cfnkm and in particular uniquely ergodic. Denote the unique probability measure of $(X,T)$ by $\mu$. By definition one has that $(Z_{k-1}(X),\mathcal{Z}_{k-1}(X),\mu_{k-1},T)$ is isomorphic as an m.p.s. to $(W_{k-1}(X),\omega_{k-1}, T)$, where $\omega_{k-1}$ is the unique ergodic measure of $(W_{k-1}(X),T)$. 

An important result of the \textit{Host-Kra structure theory} is that $\pi:Z_{k}(X) \rightarrow
Z_{k-1}(X)$, determined by $\pi_{k-1}=\pi\circ\pi_k$ (as defined in Definition \ref{def:Z_k}), is a measurable group extension w.r.t. some abelian group $A$ (See \cite[ Section 6.2]{HK05}, \cite[Chapter 9, Section 2.3]{HK2018}). By \cite[Theorem 1.1, proof of Theorem 5.3]{GL2019}, we can find a  topological model $\hat{\pi}:(\hat{Z}_k,T)\rightarrow (\hat{Z}_{k-1},T)$ of  $\pi$ which is an abelian topological group extension w.r.t. the abelian group $A$ such that $(\hat{Z}_k,T)$ is a minimal $k$-step pronilsystem and $(\hat{Z}_{k-1},T)$ is a minimal $(k-1)$-step pronilsystem.
Denote by $\phi$ and $\psi$ the measurable isomorphisms between $Z_k(X)$ and $\h{Z}_k(X)$ and $Z_{k-1}(X)$ and $\h{Z}_{k-1}(X)$ respectively. 
\[
\begin{CD}
Z_k(X) @>{\phi}>> \h{Z}_k(X)      \\
@V{\pi}VV      @VV{\h{\pi}}V\\
Z_{k-1}(X) @>{\psi }>> \h{Z}_{k-1}(X)
\end{CD}
\]
\noindent
For clarity denote $\pi_{Z_k}:=\pi_k$ from the previous paragraph.

Define $\pi_{\h{Z}_k}=\phi\circ\pi_{Z_k}$. Let $p_{k-1}:X\rightarrow W_{k-1}(X)$  be the topological canonical $(k-1)$-th projection.
Let $\pi_{\hat{Z}_{k-1}}=\h{\pi}\circ\pi_{\hat{Z}_{k}}$. 
By Corollary \ref{cor:strong max}(2), $\hat{\pi}\circ \pi_{\hat{Z}_k}$ inherits the maximality property of $\pi_{k-1}=\pi\circ \pi_{Z_k}$. By Corollary \ref{cor:strong max}(1), there exists a measurable factor map $p:\hat{Z}_{k-1}(X)\rightarrow W_{k-1}(X)$ such that $p_{k-1}=p\circ \hat{\pi}\circ \pi_{\hat{Z}_k(X)}$ a.s. As $\hat{Z}_{k-1}(X)$ is isomorphic to both $Z_{k-1}(X)$ and $W_{k-1}(X)$ as m.p.s.\footnote{Here we use that $(X,T)$ is \cfnkmp}, by Theorem \ref{thm: meas coalescence},  $p$ may be chosen to be a topological isomorphism. W.l.o.g. we will assume $p=\id$. Thus we have:

\begin{equation}\label{circ a.e.}
\text{$p_{k-1}(x)=\hat{\pi}\circ \pi_{\hat{Z}_k(X)}(x)$ for $\mu$-a.e. $x\in X$.}
\end{equation}

$$
\xymatrix@R=1.5cm{
  X\ar[d]_{\pi_{Z_k}}\ar[r]^{\id} &\ar[l]X\ar[r]^{\id}\ar[d]_{\pi_{\hat{Z}_k}}& \ar[l] X\ar[dd]^{p_{k-1}} \\
  Z_k(X)\ar@{>}[r]^{\phi}\ar[d]_{\pi}& \ar[l]\hat{Z}_k(X)\ar[d]_{\hat{\pi}} \\
  Z_{k-1}(X)\ar@{>}[r]^{\psi} &\ar[l]\hat{Z}_{k-1}(X)\ar[r]^{\id}&\ar[l]W_{k-1}(X)}
$$

\noindent
We claim that there exists a minimal subsystem $(Y,T\times T)\subset (X\times \hat{Z}_k,T\times T)$ such that $(Y,T\times T)$ is not uniquely ergodic. Assuming this, as by Theorem \ref{thm:genericity} a minimal system is uniquely ergodic if and only if every point is generic,  there exists $(x_3,u_3)\in Y$ such that $(x_3,u_3)$ is not a generic point for any measure. By Lemma \ref{lem:genericity}, there exist continuous functions $h\in C(\hat{Z}_k)$, $f\in C(X)$ such that 
\begin{equation}\label{limit no exist}\lim_{N\rightarrow \infty }\frac{1}{N}\sum_{n=1}^N h(T^nu_3)f(T^nx_3)
\end{equation}does not exist. As $(\hat{Z}_k,T)$ is a $k$-step pronilsystem, $h(T^nu_3)$ is a $k$-step nilsequence (Definition \ref{def:nilsequence}). Thus $(II)$ does not hold as required.

Our strategy in proving the claim is finding a minimal subsystem $(Y,T\times T)$ of $(X\times \hat{Z}_k,T\times T)$ which supports an invariant measure $\nu$, w.r.t which $(Y,T\times T)$ is isomorphic to $(X,\mu, T)$ as an m.p.s. We then assume for a contradiction that $(Y,T\times T)$ is uniquely ergodic.
Next we notice that the strictly ergodic system $(Y,T\times T)$, being measurably isomorphic to $(X,\mu,T)$, has $Z_k(Y)\simeq Z_k(X)$. Moreover as $(Y,T\times T)$ is a minimal subsystem of a product of the two minimal systems, $(X,T)$ and $(\hat{Z}_k,T)$, it maps onto each of them through the first, respectively second coordinate projections. From the projection on $(\hat{Z}_k,T)$, we conclude that $(Y,T)$ has a topological $k$-step pronilfactor $\hat{Z}_k$ which is measurably isomorphic to $Z_k(Y)$. By Proposition \ref{pronilfactor maximal}, one has that $(Y,T)$ is \cfnkp From the projection on $(X,T)$, we conclude by Proposition \ref{factor and k kronecker}, that $(X,T)$ is \cfnkp This constitutes a  contradiction implying  that $(Y,T)$ is not uniquely ergodic as desired.

A natural copy of $(X,\mu, T)$ inside $(X\times \hat{Z}_k,T\times T)$ is given by the \textit{graph joining} of $\pi_{\hat{Z}_k(X)}$, defined by the measure $\mu^{(k)}=(\id\times\pi_{\hat{Z}_k(X)})_*\mu :=\int \delta_x\times \delta_{\pi_{\hat{Z}_k(X)}(x)}d\mu(x)$ on $(X\times \hat{Z}_k,T)$  (see \cite[Chapter 6, Example 6.3]{G03}). Clearly
\begin{equation}\label{eq for k}\id\times\pi_{\hat{Z}_k(X)}:(X,\mathcal{X},\mu,T)\rightarrow (X\times \hat{Z}_k,\mathcal{X}\times \hat{\mathcal{Z}}_k, \mu^{(k)},T\times T),\, x\mapsto (x,\pi_{\hat{Z}_k(X)}(x)).
\end{equation}
is a measurable isomorphism and in particular $\mu^{(k)}$ is an ergodic measure of $(X\times \hat{Z}_k,T\times T)$.
However $(X\times \hat{Z}_k,\mathcal{X}\times \hat{\mathcal{Z}}_k, \mu^{(k)},T\times T)$ is a m.p.s. and not a (minimal) t.d.s. We consider an orbit closure of a $\mu^{(k)}$-generic point $(x_1,\pi_{\hat{Z}_k(X)}(x_1))$ to be determined later. By Lemma \ref{lem:supp gen}, $\mu^{(k)}$ is supported on $\OC(x_1,\pi_{\hat{Z}_k(X)}(x_1))$. However  $(\OC(x_1,\pi_{\hat{Z}_k(X)}(x_1)), T\times T)$ is not necessarily minimal. We thus pass to an (arbitrary) minimal subsystem $(Y,T\times T))\subset (\OC(x_1,\pi_{\hat{Z}_k(X)}(x_1)),T\times T))$. However $\mu^{(k)}$ is not necessarily supported on $Y$.
As explained in the previous paragraph, our final aim  will be to find (a possibly different) invariant measure $\nu\in \PMM(Y)$ which is isomorphic to $\mu$.

As $\hat{\pi}$ is a topological group extension w.r.t. the abelian group $A$, 
\begin{equation}\label{group extension}\id\times \hat{\pi}:(X\times \hat{Z}_k,T\times T)\rightarrow (X\times W_{k-1}(X),T\times T): (x,z)\mapsto (x,\hat{\pi}(z))
\end{equation}
is also a topological group extension w.r.t. the abelian group $A$. Thus $A$ acts on the fibers of $\id\times \hat{\pi}$ transitively and continuously by homeomorphisms. Moreover for all $a\in A$, $(\id\times a)_*\mu^{(k)}$ is an invariant measure on $(X\times \hat{Z}_k,T\times T)$ isomorphic to $\mu^{(k)}$ and thus isomorphic to $\mu$. We will find  $\nu\in \PMM(Y)$ of the form  $\nu=(\id\times a)_*\mu^{(k)}$. Indeed  for $\mu$-a.e. $x\in X$, $(x,\pi_{\hat{Z}_k(X)}(x))$ is a generic point of $\mu^{(k)}$. Using \eqref{circ a.e.}, one may choose $x_1\in X$ such that 
\begin{itemize}
    \item $(x_1,\pi_{\hat{Z}_k(X)}(x_1))$ is a generic point of $\mu^{(k)}$;
    \item $\hat{\pi}(\pi_{\hat{Z}_k(X)}(x_1))=p_{k-1}(x_1)$.
\end{itemize}
  From the second point it follows that: 
 $$\id\times \hat{\pi}:(\OC(x_1,\pi_{\hat{Z}_k(X)}(x_1)),T\times T)\rightarrow (\OC(x_1,p_{k-1}(x_1)),T\times T)$$is a topological factor map. As $p_{k-1}$ is a topological factor map,
 \begin{equation}\label{eq for k-1}
 \id\times p_{k-1}:(X,T)\rightarrow (\OC(x_1,p_{k-1}(x_1)),T\times T), \,x\rightarrow (x,p_{k-1}(x))
 \end{equation}
 is a topological isomorphism. Therefore $(\OC(x_1,p_{k-1}(x_1)),T\times T)$ is minimal. Thus $(\id\times \hat{\pi})_{|Y}:(Y,T)\rightarrow (\OC(x_1,p_{k-1}(x_1)),T)$ factors onto. 
 In particular there exists $z_1\in \hat{Z}_k(X)$, such that $(x_1,z_1)\in Y$ and $\hat{\pi}(z_1)=p_{k-1}(x_1)$.
 As by assumption $\hat{\pi}(\pi_{\hat{Z}_k(X)}(x_1))=p_{k-1}(x_1)$, we can find $a\in A$ such that $a.\pi_{\hat{Z}_k(X)}(x_1)=z_1$.
As $(x_1,\pi_{\hat{Z}_k(X)}(x_1))$ is a generic point of $\mu^{(k)}$, it follows that $(x_1,a.\hat{\pi}_k(x_1))=(x_1,z_1)$ is a generic point of $\nu:=(\id\times a)_*\mu^{(k)}$. Therefore by Lemma \ref{lem:supp gen}, the invariant measure $\nu\simeq \mu$ is supported on the minimal subsystem $\OC{(x_1,z_1)}=Y$. This ends the proof. 
\end{proof}

\specialsectioning
\bibliographystyle{alpha}
\bibliography{references}

\Addresses
\end{document}